\documentclass[reqno,10pt]{amsart}
\usepackage{amsfonts, amsmath, amssymb, amsthm, amsrefs, color}

\allowdisplaybreaks[1]

\makeatletter
\numberwithin{equation}{section}
\numberwithin{figure}{section}
\theoremstyle{plain}
\newtheorem{thm}{\protect\theoremname}
  \theoremstyle{definition}
  
  \theoremstyle{remark}
  
  \theoremstyle{plain}
  \newtheorem{lem}[thm]{\protect\lemmaname}
  \theoremstyle{plain}
  \newtheorem{prop}[thm]{\protect\propositionname}

\newcommand{\eps}{\varepsilon}

\makeatother

\newcommand{\e}{\varepsilon}
 \renewcommand{\(}{\left(}
\renewcommand{\)}{\right)}
\newcommand{\M}{\mathcal{M}}
\newcommand{\g}{\mathfrak g}

  \providecommand{\definitionname}{Definition}
  \providecommand{\lemmaname}{Lemma}
  \providecommand{\propositionname}{Proposition}
  \providecommand{\remarkname}{Remark}
\providecommand{\theoremname}{Theorem}

\begin{document}

\title[Blowing-up solutions ...]{Blowing-up solutions concentrating along minimal submanifolds for some supercritical
elliptic problems on Riemannian manifolds}

\author{Marco Ghimenti, Anna Maria Micheletti, Angela Pistoia}

\maketitle

\begin{abstract}
  Let $(M,g)$ and $(K,\kappa)$ be two Riemannian manifolds of dimensions $m$ and $k ,$ respectively. Let   $\omega\in C^2(N),$ $\omega> 0.$
  The warped product $ M\times _\omega K$ is the $ (m+k)$-dimensional product manifold
$M\times K$ furnished with  metric   $  g+\omega^2 \kappa.$
We prove that    the  supercritical problem
$$-\Delta _{g+\omega^2 \kappa}u+h u=u^{ {m+2\over m-2} \pm\e},\ u>0,\ \hbox{in}\ (M\times _\omega K,g+\omega^2 \kappa)$$
has a solution which concentrate along a $k$-dimensional minimal submanifold $\Gamma$ of $M\times _\omega N$ as the real parameter $\e$ goes to zero, provided the function $h$  and the sectional curvatures
along $\Gamma$ satisfy a suitable condition.
 \end{abstract}



\section{Introduction and statement of main results}\label{intro}

We deal with the semilinear elliptic equation
\begin{equation}\label{p}
 -\Delta _\g u+h u=u^{p-1},\ u>0,\ \hbox{in}\ (\M,\g)
 \end{equation}
where   $(\M,\g)$  is a $n-$dimensional compact Riemannian manifold without boundary, $h$ is a $C^1-$real function on $\M$   s.t. $-\Delta_\g+h$  is coercive
and $p>2$.

 The compactness of the embedding $H^1_\g(\M)\hookrightarrow L^p_\g(\M)$   for any $p\in(2,2^*_n)$
where $2^*_n:={2n\over n-2}$ if  $n\ge3$ and $ 2^*_n:=+\infty$ if $n=2,$ ensures that
$$\inf\limits_{u\in H^1_\g(\M)\atop u\not=0}{\int\limits_\M\(|\nabla_\g u|^2+h u^2\)d\mu_\g\over \(\int\limits_\M | u|^pd\mu_\g\)^{2/p}} $$
 is achieved and so problem \eqref{p}  has always a solution for any $p\in(2,2^*_n).$

In the critical case, i.e.  $p=2^*_n,  $  the situation turns out to be more complicated. In particular, the existence of solutions is related to the position of the potential $h $ with respect to the geometric potential $h_g:= {n-2\over 4(n-1)} S_\g$, where $S_\g$ is the scalar curvature of the manifold.
If $h\equiv h_\g$, then problem \eqref{p}  is referred to as  the   Yamabe problem and it has always a solution (see Aubin \cites{Aub1,Aub2}, Schoen \cite{Sch1}, Trudinger \cite{Tru}, and Yamabe \cite{Yam} for early references on the subject). When $h< h_\g$ somewhere in $\M$, existence of a solution is guaranteed by a minimization argument (see for example Aubin \cites{Aub1,Aub2}). The situation is extremely delicate when  $h \ge h_\g$ because blow-up phenomena can occur as pointed out by Druet in \cite{D1,D2}.\\
The supercritical case $p>2^*_n$ is even more difficult to deal with.
A first result in this direction is a perturbative result   due to Micheletti, Pistoia and V\'etois \cite{MPV}.
They consider the almost critical problem \eqref{p} when
$p=2^*_n \pm\varepsilon$, i.e.
if $p=2^*_n-\varepsilon$ the problem \eqref{p} is slightly subcritical and if $p=2^*_n+\varepsilon$ the problem \eqref{p} is slightly supercritical. They prove the following results.
\begin{thm}\label{mipive}  Assume
 $n\ge6  $ and  $\xi_0\in \M$ is a non degenerate critical point of $h-{n-2\over 4(n-1)}S_\g.$
Then
 \begin{itemize}
 \item[(i)] if $ h(\xi_0)>{n-2\over 4(n-1)}S_\g(\xi_0)$ then   the slightly subcritical   problem \eqref{p} with $p=2^*_{n}-1-\e,$ has a solutions $u_\e$ which concentrates  at $\xi_0,$
 \item[(ii)] if $ h(\xi_0)<{n-2\over 4(n-1)}S_\g(\xi_0)$ then the slightly supercritical problem \eqref{p} with $p=2^*_{n}-1-\e,$ has a solutions $u_\e$ which concentrates  at $\xi_0$ as $\e\rightarrow 0$.
\end{itemize}
\end{thm}

Now, for any integer $0\le k\le n-3 $ let $2^*_{n,k}={2(n-k)\over n-k-2} $  be the \textit{$(k+1)-$st critical exponent}.
We remark that $2^*_{n,k}=2^*_{n-k,0}$   is nothing but the critical exponent for the Sobolev embedding  ${\mathrm H}^ 1_\g(\M)\hookrightarrow {\mathrm L}^{q}_\g(\M),$  when $(\M,\g)$ is a $(n-k)-$dimensional Riemannian manifold.
In particular,  $ 2^*_{n,0}={2 n\over n- 2}$    is the usual Sobolev critical exponent.\\
We can summarize the results proved by Micheletti, Pistoia and V\'etois just saying that  problem \eqref{p} when $p\to 2^*_{n,0}$ (i.e. $k=0$) has positive  solutions   blowing-up at points.
    Note that   a  point  is a   $0-$dimensional manifold!  Therefore, a natural question arises:
  \textit{ does  problem \eqref{p} have solutions blowing-up at  $k-$dimensional submanifolds   when $p\to 2^*_{n,k}$?}\\

In the present paper,   we give a positive answer
when $(\M,\g)$ is a warped product manifold.

We recall the notion of warped product introduced by Bishop and O'Neill in \cite{BO}.
Let $(M,g)$ and $(K,\kappa)$ be two riemannian manifolds of dimensions $m$ and $k ,$ respectively. Let   $\omega\in C^2(M),$ $\omega> 0$ be a differentiable function.  The warped product $\M = M\times _\omega K$ is the product (differentiable) $n-$dimensional ($n:=m+k$) manifold
$M\times K$ furnished with the riemannian   metric is $\g=g +\omega^2\kappa.$ $\omega$ is called {\em warping function}.
For example, every surface of revolution (not crossing the axis of revolution) is
isometric to a warped product, with $M$ the generating curve, $K=S^1$ and $\omega(x)$ the distance from $x \in M$ to the axis of revolution.

It is not difficult to check that if $u\in C^2(M\times _\omega K)$   then
\begin{equation}\label{equ2}\Delta _\g u=\Delta _{g } u +{m \over \omega}g \(\nabla_g f, \nabla _g u\)+{1\over \omega^2}\Delta _{\kappa} u.\end{equation}
Assume $h$ is invariant with respect to $K,$ i.e. $h(x,y)=h(x)$ for any $(x,y)\in M\times K.$
If we look for solutions to \eqref{p} which are invariant with respect to $K,$ i.e. $u(x,y)=v(x)$ then by \eqref{equ2}
we immediately deduce that $u$ solves \eqref{p} if and only if $v$ solves
\begin{equation}\label{equ3}-\Delta _{g } v -{m \over \omega}g \(\nabla _{g } f, \nabla  _{g } v\)+h v=v^{p-1} \quad\hbox{in}\ (M,g ).
\end{equation}
or equivalently
\begin{equation}\label{equ4}- \mathrm{div}_{g }\( {\omega^m}\nabla _{g } v\)+ {\omega^m}h v ={\omega^m}v^{p-1},\ v>0\quad \hbox{in}\ (M,g ).
\end{equation}
Here we are interested in studying problem \eqref{equ4} when the exponent $p$ approaches the  higher critical exponent $2^*_{n,k}=2^*_{m}$, i.e. $p=2^*_{m}-\varepsilon$ for some small real parameter $\varepsilon.$
It is clear that if $v$ is a solution to problem \eqref{equ3} which concentrates at a point $\xi_0\in M$ then $u(x,y)=v(x)$ is a solution to problems \eqref{p}
which concentrates along the fiber $\{\xi_0\}\times K$, which is a $k -$dimensional submanifold of $M.$
It is important to notice the the fiber $\{\xi_0\}\times K$ is   totally geodesic in $ M\times _\omega K$ (and in particular a minimal submanifold of $ M\times _\omega K$) if $\xi_0$ is a critical point of the warping function $\omega.$\\

Therefore, we are lead to study the more general anisotropic almost critical problem
\begin{equation}\label{equ5}- \mathrm{div}_{g }\( a(x)\nabla _{g } u\)+ {a(x)}h u ={a(x)}u^{{m+2\over m-2}-\varepsilon},\ u>0\quad \hbox{in}\ (M,g )
\end{equation}
where $(M,g)$ is a $m-$dimensional compact Riemannian manifolds, $a\in C^2(M)$ with $\min_M a>0,$ $h\in C^2(M)$ such that the anisotropic operator $- \mathrm{div}_{g }\( a(x)\nabla _{g } u\)+ {a(x)}h u $ is coercive
and $\varepsilon\in\mathbb R.$

Our main result reads as follows.
\begin{thm}\label{main}
 Assume  $m\ge9  $ and
 $\xi_{0}\in M$ is a non degenerate critical point of $a$.  Then
 \begin{itemize}
 \item[(i)] if $ h(\xi_0)>{m-2\over 4(m-1)}S_g(\xi_0)-{3 (m-2)\over 2(m-1)}{\Delta_g a (\xi_0)\over  a(\xi_0)}$ then  if $\varepsilon>0$ is small enough the slightly subcritical   problem \eqref{equ5} has a solutions $u_\e$ which concentrates at $\xi_0$  as $\e\rightarrow 0$,
 \item[(ii)] if $ h(\xi_0)<{n-2\over 4(n-1)}S_g(\xi_0)-{3 (m-2)\over 2(m-1)}{\Delta_ga(\xi_0)\over a (\xi_0)}$ then  if $\varepsilon<0$ is small enough  then the slightly supercritical problem \eqref{equ5}   has a solutions $u_\e$ which concentrates  at $\xi_0$  as $\e\rightarrow 0$.
\end{itemize}
\end{thm}
 In particular, Theorem \ref{main} applies to the case $a=\omega^m$ where $\omega$ is the warping function. We recall that if  $\xi_0$ is a critical point of the warping function $\omega,$ then the fiber $\Gamma:=\{\xi_0\}\times K$ is  a minimal $k-$dimensional submanifold of the warped product manifold $\M\times_\omega K $ equipped with the metric $\g=g+\omega^2\kappa.$ Let
 $$\Sigma_\g(\Gamma ):= {m-2\over 4(m-1)}S_g(\xi_0 )-{3m(m-2)\over 2(m-1)}{\Delta_g \omega (\xi_0)\over\omega (\xi_0)},$$
 which turns out to be a weighted mean of sectional curvatures of   $\Gamma$.
From the above discussion and Theorem \ref{main} we immediately deduce the following result concerning the supercritical problem \eqref{p}.
 \begin{thm}\label{main1} Assume $m\ge9$,  $h$ is invariant with respect to $K $ and $p=2^*_m-\varepsilon.$

 \begin{itemize}
 \item[(i)] if $ h(\Gamma )>\Sigma_\g(\Gamma ),$   then  if $\varepsilon>0$ is small enough the supercritical  problem \eqref{p} has a solutions $u_\e,$ invariant with respect to $K,$ which concentrates along  $\Gamma$  as $\e\rightarrow 0$,
 \item[(ii)] if $ h(\Gamma )<\Sigma_\g(\Gamma ),$   then  if $\varepsilon<0$ is small enough  then the   supercritical problem \eqref{p}   has a solutions $u_\e$ invariant with respect to $K,$ which concentrates along  $\Gamma$   as $\e\rightarrow 0$.
\end{itemize}
\end{thm}

Let us state some open problems about the anisotropic problem \eqref{equ5}.

\begin{itemize}
\item[(a)] Theorem \ref{main} holds true when $m\ge9.$
The interesting question concerns the low dimensions $m=3,\dots,8.$ For example, one could ask if  the results obtained by Druet \cite{D1,D2} for  the Yamabe problem in low dimensions are true anymore.
\item[(b)] Theorem \ref{main} holds true when the potential $h$ is different from the geometric potential $\Sigma(\Gamma)$.
It is interesting to see what happens when $h$ coincides somewhere with  $\Sigma(\Gamma)$.  In particular, it could be interesting to obtain similar results to the ones obtained by Esposito-Pistoia-Vet\'ois \cite{epv} for  the Yamabe problem.
\item[(c)] Theorem \ref{main} concerns the case $p$ close to $2^*_m$ but different from it. Is it possible to establish an existence result for the pure critical case $p=2^*_m$ similar to the results obtained for the Yamabe problem?
    \end{itemize}

Finally, we ask if similar results hold true  in the non symmetric case. More precisely,  if $p\to 2^*_{m,k}$ for some integer $k\ge1$ does there exist a solution to problem \eqref{p} which blows-up along a minimal $k-$dimensional submanifold $\Gamma$ provided the potential $h$ is different from a suitable
  weighted mean of sectional curvatures of   $\Gamma$?  Recently, Davila-Pistoia-Vaira \cite{dpv} gave a positive answer when $k=1$.
\\

The paper is organized as follows. In Section 2 we introduce the necessary framework.   In
Section 4 we reduce the problem to a finite dimensional one, via a   Ljapunov-Schmidt reduction
and in    Section 5 we study the finite dimensional problem and we prove Theorem \ref{main}.

\section{Setting of the problem}

Let $H$ the Hilbert space $H_{g}^{1}(M)$ endowed with the scalar
product
\[
\left\langle u,v\right\rangle _{H}=\int_{M}a(x)\nabla_{g}u\nabla_{g}vd\mu_{g}+\int_{M}a(x)h(x)uvd\mu_{g}.
\]
Let $\|\cdot\|_{H}$ be the norm induced
by $\left\langle \cdot,\cdot\right\rangle _{H}$, which is  equivalent to the
usual one.
We also denote the usual $L^{q}$-norm of a function $u\in L_{g}^{q}(M)$
by
\[
|u|_{g}=\left(\int_{M}|u|^{q}d\mu_{g}\right)^{1/q}.
\]
Let $i_{H}^{*}:L_g^{\frac{2m}{m+2}}(M)\rightarrow H$ be the adjoint operator
of the embedding $i:H\rightarrow L^{2_{m}^{*}}_g(M)$, that is
\begin{eqnarray*}
u=i_{H}^{*}(v) & \Leftrightarrow & \left\langle i_{H}^{*}(v),\varphi\right\rangle _{H}=\int_{M}v\varphi d\mu_{g}\ \ \ \forall\varphi\in H\\
 & \Leftrightarrow & \int_{M}a(x)\nabla_{g}u\nabla_{g}\varphi d\mu_{g}+\int_{M}a(x)h(x)u\varphi d\mu_{g}=\int_{M}v\varphi d\mu_{g}\ \ \ \forall\varphi\in H
\end{eqnarray*}

We recall (see \cite{MPV}) the following
inequality
\begin{equation}\label{Rem:GT}
\left|i_{H}^{*}(w)\right|_{s}\le C\left|w\right|_{\frac{ms}{m+2s}}\text{ for }w\in L^{\frac{ms}{m+2s}}
\end{equation}
 where $s>\frac{2m}{m-2}=2_{m}^{*}$ so that $\frac{ms}{m+2s}>\frac{2m}{m+2}$.

We consider the Banach space ${H}_{\varepsilon}=H_{g}^{1}(M)\cap L_{g}^{s_{\varepsilon}}(M)$
with the norm
\[
\|u\|_{H_\varepsilon}=\|u\|_{H}+\left|u\right|_{s_{\varepsilon}}
\]
where we set
\[
s_{\varepsilon}=\left\{ \begin{array}{cc}
2_{m}^{*}-\frac{m}{2}\varepsilon & \text{ if }\varepsilon<0\\
2_{m}^{*} & \text{ if }\varepsilon>0
\end{array}\right.
\]
We remark that in the subcritical case $\varepsilon>0$ the space
$ {H}_{\varepsilon}$ is nothing but the space $H_{g}^{1}(M)$ with norm
$\|\cdot\|_{H}$.
By \eqref{Rem:GT}, we easily deduce that
\begin{equation}
\|i_{H}^{*}(w)\|_{H_{\varepsilon}}\le c|w|_{\frac{ms_{\varepsilon}}{m+2s_{\varepsilon}}}\text{ if }w\in L^{\frac{ms_{\varepsilon}}{m+2s_{\varepsilon}}}.\label{eq1.4a}
\end{equation}
Finally,  we can rewrite Equation (\ref{equ5}) as
\begin{equation}
u=i_{H}^{*}(a(x)f_{\varepsilon}(u))\ \ \ u\in\mathcal{H}_{\varepsilon}\label{eq:1.4}
\end{equation}
where $f_{\varepsilon}(u)=(u^{+})^{2^{*}-1-\varepsilon}$ and
$u^{+}=\max\{u,0\}$.\\

Now, let us introduce the main ingredient to build a solution to problem \eqref{eq:1.4}, namely the standard bubble
$$
U(z) :=\frac{\alpha_{m}}{\left(1+|z|^{2}\right)^{\frac{m-2}{2}}}\ \ z\in\mathbb{R}^{m}
$$
where $\alpha_{m}:=(m(m-2))^{\frac{m-2}{4}}$. It is well known that the functions
$$U_{\delta,y}(z):=\delta^{-\frac{m-2}{2}}U\( {z-y\over\delta}\), z,y\in\mathbb R^m,\ \delta>0$$
are all the positive solutions to the  critical problem $-\Delta U=U^{2_{m}^{*}-1}$ on $\mathbb R^m.$

 We are going to read the euclidean bubble $U_{\delta,y}$ on the manifold $M$ via geodesic coordinates.
Let $\chi$ be a smooth cut-off function such that $0\le\chi\le1$,
$\chi(z)=1$ if $z\in B(0,r/2)\subset\mathbb{R}^{m}$, $\chi(z)=0$
if $z\in\mathbb{R}^{m}\smallsetminus B(0,r)$, $|\nabla\chi|\le2/r$,
where $r$ is the iniectivity radius of $M$.
Let us define on $M$ the function
\begin{equation}\label{eq:W}
W_{\delta,\eta}  (x)=\left\{ \begin{array}{ccc}
\chi(\exp_{\xi_{0}}^{-1}(x))\delta^{-\frac{m-2}{2}}U(\delta^{-1}\exp_{\xi_{0}}^{-1}(x)-\eta) &  & \text{if }x\in B_{g}(\xi_{0},r)\\
0 &  & \text{if }x\in M\smallsetminus B_{g}(\xi,r)
\end{array}\right..
\end{equation}

We will look for a solution of (\ref{eq:1.4}) or, equivalently of
(\ref{equ5}), as
$
u=W_{\delta,\eta}+\Phi,
$ where
\begin{equation}
\label{delta}
\delta=\delta_{\varepsilon}(t)=\sqrt{|\varepsilon|t}\ \hbox{for some}\ t>0\ \hbox{and}\ \eta\in\mathbb R^m.
\end{equation}
We remark that as $\varepsilon $ goes to $0$ the function $W_{\delta,\eta}$ blows-up at the point $\xi_0$.
The remainder term $\Phi$ belongs to  the space $K^\perp_{\delta,\eta}$, which
is introduced as follows. It is well known  that any
solution to the linearized equation
\[
-\Delta\psi=(2_{m}^{*}-1)U^{2_{m}^{*}-2}\psi
\]
is a linear combination of the functions
\begin{eqnarray*}
V_{0}(z)=\left.\frac{\partial}{\partial\delta}\left[\delta^{-\frac{m-2}{2}}U(\delta^{-1}z)\right]\right|_{\delta=1} &  & V_{i}(z)=\frac{\partial U}{\partial z_{i}}(z)\ i=1,\dots,m.
\end{eqnarray*}
Let us define on $M$ the functions
\[
Z_{\delta,\eta}^{i}(x)=\left\{ \begin{array}{ccc}
\chi(\exp_{\xi_{0}}^{-1}(x))\delta^{-\frac{m-2}{2}}V_{i}(\delta^{-1}\exp_{\xi_{0}}^{-1}(x)-\eta) &  & \text{if }x\in B_{g}(\xi_{0},r)\\
0 &  & \text{if }x\in M\smallsetminus B_{g}(\xi,r)
\end{array}\right..
\]
Let us introduce the spaces $K_{\delta,\eta}=\text{span}\left\langle Z_{\delta,\eta}^{0},\dots,Z_{\delta,\eta}^{m}\right\rangle $
and
\[
K_{\delta,\eta}^{\bot}=\left\{ \Phi\in H_\varepsilon\ :\ \left\langle \Phi,Z_{\delta,\eta}^{i}\right\rangle _{H}=0\text{ for }i=1,\dots,m\right\} .
\]
In order to solve problem (\ref{eq:1.4}) we will solve the couple
of equations
\begin{equation}
\pi_{\delta,\eta}\left(W_{\delta,\xi}+\Phi-i_{H}^{*}(af_{\varepsilon}(W_{\delta,\eta}+\Phi))\right)=0\label{eq:1.6}
\end{equation}
\begin{equation}
\pi_{\delta,\eta}^{\bot}\left(W_{\delta,\eta}+\Phi-i_{H}^{*}(af_{\varepsilon}(W_{\delta,\eta}+\Phi))\right)=0\label{eq:1.7}
\end{equation}
where $\pi_{\delta,\eta}:H_\varepsilon\rightarrow K_{\delta,\eta}$
and $\pi_{\delta,\eta}^{\bot}:H_\varepsilon\rightarrow K_{\delta,\eta}^{\bot}$
are the orthogonal projection and $\Phi\in\mathcal{H}_{\varepsilon}\cap K_{\delta,\eta}^{\bot}$.

\section{The finite dimensional reduction}

First of all, we   solve equation (\ref{eq:1.7}).
We set
\begin{align*}
L_{\varepsilon,\delta,\eta} & (\Phi):=\pi_{\delta,\eta}^{\bot}\left\{ \Phi-i_{H}^{*}\left(a(x)f_{\varepsilon}'(W_{\delta,\eta})[\Phi]\right)\right\} \\
N_{\varepsilon,\delta,\eta} & (\Phi):=\pi_{\delta,\eta}^{\bot}\left\{ i_{H}^{*}\left(a(x)\left(f_{\varepsilon}(W_{\delta,\eta}+\Phi)-f_{\varepsilon}(W_{\delta,\eta})-f_{\varepsilon}'(W_{\delta,\eta})[\Phi]\right)\right)\right\} \\
R_{\varepsilon,\delta,\eta} & :=\pi_{\delta,\eta}^{\bot}\left\{ i_{H}^{*}\left(a(x)f_{\varepsilon}(W_{\delta,\eta})\right)-W_{\delta,\eta}\right\}
\end{align*}
so equation (\ref{eq:1.7}) reads as
\begin{equation}
L_{\varepsilon,\delta,\eta}(\Phi)=N_{\varepsilon,\delta,\eta}(\Phi)+R_{\varepsilon,\delta,\eta}\label{eq:1.8}
\end{equation}

\begin{lem}
For any real numbers $\alpha$ and $\beta$ with $0<\alpha<\beta$,
there exists a positive constant $C_{\alpha,\beta}$ such that, for
$\varepsilon$ small, for any $\eta\in\mathbb{R}^{m}$, any real number
$t\in[\alpha,\beta]$ and any $\Phi\in\mathcal{H}_{\varepsilon}\cap K_{\delta_{\varepsilon}(t),\eta}^{\bot}$
it holds
\begin{equation}
\|L_{\varepsilon,\delta_{\varepsilon}(t),\eta}(\Phi)\|_{H,s_{\varepsilon}}\ge C_{\alpha,\beta}\|\Phi\|_{H,s_{\varepsilon}}\label{eq:1.10}
\end{equation}
\end{lem}
\begin{proof}
The proof is the same of {[}\cite{MPV} Lemma 3.1{]} which we refers to\end{proof}
\begin{lem}
\label{lem:3.2}If $m\ge9$, for any real numbers $\alpha$ and $\beta$
with $0<\alpha<\beta$, there exists a positive constant $C_{\alpha,\beta}$
such that, for $\varepsilon$ small enough, for any $\eta\in\mathbb{R}^{m}$,
any real number $t\in[\alpha,\beta]$
\[
\|R_{\varepsilon,\delta_{\varepsilon}(t),\eta}\|_{H,s_{\varepsilon}}\le C_{\alpha,\beta}|\varepsilon|\left|\log|\varepsilon|\right|
\]
\end{lem}
\begin{proof}
By definition of $i_{H}^{*}$ and  \eqref{eq1.4a}
\[
\|R_{\varepsilon,\delta_{\varepsilon}(t),\eta}\|_{H,s_{\varepsilon}}\le c\|a(x)f_{\varepsilon}(W_{\delta_{\varepsilon}(t),\eta})+a(x)\Delta_{g}W_{\delta_{\varepsilon}(t),\eta}+\nabla_{g}a(x)\nabla_{g}W_{\delta_{\varepsilon}(t),\eta}-a(x)hW_{\delta_{\varepsilon}(t),\eta}\|_{\frac{ms_{\varepsilon}}{m+2s_{\varepsilon}}}
\]
Using {[}\cite{MPV}, Lemma 3.2{]}, by direct computation it is easy to prove
that
\[
\|f_{\varepsilon}(W_{\delta_{\varepsilon}(t),\eta})+\Delta_{g}W_{\delta_{\varepsilon}(t),\eta}-hW_{\delta_{\varepsilon}(t),\eta}\|_{\frac{ms_{\varepsilon}}{m+2s_{\varepsilon}}}\le C_{\alpha,\beta}|\varepsilon|\left|\log|\varepsilon|\right|.
\]
It remains to estimate the term
\[
\|\nabla_{g}a\nabla_{g}W_{\delta_{\varepsilon}(t),\eta}\|_{\frac{ms_{\varepsilon}}{m+2s_{\varepsilon}}}.
\]
Since $\xi_{0}$ is a critical point for the function $a$ we have
\begin{align*}
\|\nabla_{g}a\nabla_{g}W_{\delta_{\varepsilon}(t),\eta}\|_{q}^{q} & \le C\max_{1\le j\le m}\int_{|y|<r}|y|^{q}\left|\frac{\partial}{\partial y_{j}}\left(\chi(y)\delta^{-\frac{m-2}{2}}U\left(\frac{y}{\delta}-\eta\right)\right)\right|^{q}|g(y)|^{\frac{1}{2}}dy\\
 & \le C\max_{1\le j\le m}\int_{|y|<r}|y|^{q}\left|\delta^{-\frac{m-2}{2}}\frac{\partial}{\partial y_{j}}\chi(y)U\left(\frac{y}{\delta}-\eta\right)\right|^{q}dy\\
& + C\max_{1\le j\le m}\int_{|y|<r}|y|^{q}\left|\delta^{-\frac{m-2}{2}}\chi(y)\frac{\partial}{\partial y_{j}}U\left(\frac{y}{\delta}-\eta\right)\right|^{q}dy\\
 & \le C\max_{1\le j\le m}\int_{{r\over2\delta}\le |z|\le{r\over\delta}}|\delta z|^{q}\left|\delta^{-\frac{m-2}{2}}U\left(z-\eta\right)\right|^{q}\delta^{m}dz\\
 & + C\max_{1\le j\le m}\int_{|z|\le{r\over\delta}}|\delta z|^{q}\left|\delta^{-\frac{m}{2}}\frac{\partial U}{\partial z_{j}}\left(z-\eta\right)\right|^{q}\delta^{m}dz\\
 & \le C\delta^{\frac{m-2}{2} q}+C\delta^{m+q -\frac{m}{2} q}.
\end{align*}
Therefore,
\begin{equation}
\|\nabla_{g}a\nabla_{g}W_{\delta_{\varepsilon}(t),\eta}\|_{q}\le\delta^{2},\label{eq:4.3}
\end{equation}
because $m\ge9$ and if $q=\frac{ms_{\varepsilon}}{m+2s_{\varepsilon}}$ then $1-\frac{m}{2}+\frac{m}{q}=2$
if $\varepsilon>0$ or $1-\frac{m}{2}+\frac{m}{q}$ close to 2 if
$\varepsilon<0$ and sufficiently small. This concludes the proof.\end{proof}
\begin{prop}
\label{prop:Phi}If $m\ge9$, for any real numbers $\alpha$ and $\beta$
with $0<\alpha<\beta$, there exists a positive constant $C_{\alpha,\beta}$
such that, for $\varepsilon$ small enough, for any $\eta\in\mathbb{R}^{m}$,
any real number $t\in[\alpha,\beta]$, there exists a unique solution $\Phi_{\delta_{\varepsilon}(t),\eta}\in {H}_{\varepsilon}\cap K_{\delta_{\varepsilon}(t),\eta}^{\bot}$
of equation (\ref{eq:1.7}) such that
\[
\|\Phi_{\varepsilon,\delta_{\varepsilon}(t),\eta}\|_{H,s_{\varepsilon}}\le C_{\alpha,\beta}|\varepsilon|\left|\log|\varepsilon|\right|.
\]
Moreover $\Phi_{\varepsilon,\delta_{\varepsilon}(t),\eta}$ is continuously
differentiable with respect to $t$ and $\eta$.\end{prop}
\begin{proof}
In order to solve equation (\ref{eq:1.7}) we look for a fixed point
for the operator
\[
T_{\varepsilon,\delta_{\varepsilon}(t),\eta}=T: {H}_{\varepsilon}\cap K_{\delta_{\varepsilon}(t),\eta}^{\bot}\rightarrow {H}_{\varepsilon}\cap K_{\delta_{\varepsilon}(t),\eta}^{\bot}
\]
defined by
\[
T_{\varepsilon,\delta_{\varepsilon}(t),\eta}(\Phi)=L_{\varepsilon,\delta_{\varepsilon}(t),\eta}^{-1}\left\{ N_{\varepsilon,\delta_{\varepsilon}(t),\eta}(\Phi)+R_{\varepsilon,\delta_{\varepsilon}(t),\eta}\right\} .
\]
We have that
\begin{align*}
\|T_{\varepsilon,\delta_{\varepsilon}(t),\eta}(\Phi)\|_{H,s_{\varepsilon}} & \le\|N_{\varepsilon,\delta_{\varepsilon}(t),\eta}(\Phi)\|_{H,s_{\varepsilon}}+\|R_{\varepsilon,\delta_{\varepsilon}(t),\eta}\|_{H,s_{\varepsilon}}\\
\|T_{\varepsilon,\delta_{\varepsilon}(t),\eta}(\Phi_{1})-T_{\varepsilon,\delta_{\varepsilon}(t),\eta}(\Phi_{2})\|_{H,s_{\varepsilon}} & \le\|N_{\varepsilon,\delta_{\varepsilon}(t),\eta}(\Phi_{1})-N_{\varepsilon,\delta_{\varepsilon}(t),\eta}(\Phi_{2})\|_{H,s_{\varepsilon}}.
\end{align*}
Since $m\ge9$ a simple application of mean value theorem gives
\[
|f_{\varepsilon}(x+y)-f_{\varepsilon}(x+z)-f'_{\varepsilon}(x)(y-z)|\le C|y-z|\left[|y|+|z|\right]\left[x+|y|+|z|\right]^{2_{m}^{*}-3-\varepsilon}
\]
 for all $x>0,$ $x,y,z\in\mathbb{R}.$ Thus
\begin{align}
&\|N_{\varepsilon,\delta_{\varepsilon}(t),\eta}(\Phi_{1})-
N_{\varepsilon,\delta_{\varepsilon}(t),\eta}(\Phi_{2})\|_{H,s_{\varepsilon}}\nonumber\\ &
\le\|f_{\varepsilon}(W_{\delta_{\varepsilon}(t),\eta}+\Phi_{1})-f_{\varepsilon}(W_{\delta_{\varepsilon}(t),\eta}+
\Phi_{2})-f'_{\varepsilon}(W_{\delta_{\varepsilon}(t),\eta})(\Phi_{1}-\Phi_{2})\|_{\frac{ms_{\varepsilon}}
{m+2s_{\varepsilon}}}\nonumber \\ &
\le C\|\Phi_{1}-\Phi_{2}\|_{\beta_{\varepsilon}}\left[\|\Phi_{1}\|_{\beta_{\varepsilon}}
+\|\Phi_{2}\|_{\beta_{\varepsilon}}\right]
\times\left[\|W_{\delta_{\varepsilon}(t),\eta}\|_{\beta_{\varepsilon}}+\|\Phi_{1}\|_{\beta_{\varepsilon}}+\|\Phi_{2}\|_{\beta_{\varepsilon}}\right]^{2_{m}^{*}-3-\varepsilon}\label{eq:3.1}
\end{align}
where $\beta_{\varepsilon}=\frac{ms_{\varepsilon}}{m+2s_{\varepsilon}}(2^{*}-1-\varepsilon)$,
so $\beta_{\varepsilon}=s_{\varepsilon}$ if $\varepsilon<0$ or $\beta_{\varepsilon}=2_{m}^{*}-\varepsilon\frac{2m}{m+2}$
if $\varepsilon>0$.

In particular

\begin{align}
\|N_{\varepsilon,\delta_{\varepsilon}(t),\eta}(\Phi_{1})\|_{H,s_{\varepsilon}}&
\le\|f_{\varepsilon}(W_{\delta_{\varepsilon}(t),\eta}+\Phi_{1})-f_{\varepsilon}(W_{\delta_{\varepsilon}(t),\eta})
-f'_{\varepsilon}(W_{\delta_{\varepsilon}(t),\eta})(\Phi_{1})\|_{\frac{ms_{\varepsilon}}{m+2s_{\varepsilon}}}\nonumber\\
&
\le C\|\Phi_{1}\|_{\beta_{\varepsilon}}^{2}\left[\|W_{\delta_{\varepsilon}(t),\eta}\|_{\beta_{\varepsilon}}
+\|\Phi_{1}\|_{\beta_{\varepsilon}}\right]^{2_{m}^{*}-3-\varepsilon}.\label{eq:3.2}
\end{align}
Thus by (\ref{eq:3.2}) and by Lemma \ref{lem:3.2} we have that
\[
\|T_{\varepsilon,\delta_{\varepsilon}(t),\eta}(\Phi)\|_{H,s_{\varepsilon}}\le\|\Phi\|_{H,s_{\varepsilon}}^{2}+C_{\alpha,\beta}|\varepsilon|\left|\log|\varepsilon|\right|,
\]
so if $\|\Phi\|_{H,s_{\varepsilon}}\le2C_{\alpha,\beta}|\varepsilon|\left|\log|\varepsilon|\right|$
and for $\varepsilon$ small $\|T_{\varepsilon,\delta_{\varepsilon}(t),\eta}(\Phi)\|_{H,s_{\varepsilon}}\le2C_{\alpha,\beta}|\varepsilon|\left|\log|\varepsilon|\right|$.
Moreover by (\ref{eq:3.1}) we have
\[
\|T_{\varepsilon,\delta_{\varepsilon}(t),\eta}(\Phi_{1})-T_{\varepsilon,\delta_{\varepsilon}(t),\eta}(\Phi_{2})\|_{H,s_{\varepsilon}}\le K\|\Phi_{1}-\Phi_{2}\|_{H,s_{\varepsilon}}
\]
 for some $K<1$ if $\|\Phi_{i}\|_{H,s_{\varepsilon}}\le2C_{\alpha,\beta}|\varepsilon|\left|\log|\varepsilon|\right|$
and $\varepsilon$ small enough.

Therefore, a contraction mapping argument proves that the map $T_{\varepsilon,\delta_{\varepsilon}(t),\eta}$
admits a fixed point $\Phi_{\delta_{\varepsilon}(t),\eta}$. The regularity
of $\Phi_{\delta_{\varepsilon}(t),\eta}$ with respect to $\eta$
and $t$ follows by standard arguments using the implicit function
theorem.
\end{proof}

\section{The reduced problem and proof of Theorem \ref{main} }

Let $J_{\varepsilon}:H_\varepsilon\rightarrow\mathbb{R}$ be the energy associated to problem (\ref{equ5}) defined
by
\begin{equation}
J_{\varepsilon}(u)=\frac{1}{2}\int_{M}a(x)\left(|\nabla_{g}u|^{2}+h(x)u^{2}\right)d\mu_{g}-\frac{1}{2_{m}^{*}-\varepsilon}\int_{M}a(x)\left(u^{+}\right)^{2_{m}^{*}-\varepsilon}d\mu_{g}.\label{eq:Je}
\end{equation}
 It is well known that any critical point of
$J_{\varepsilon}$ is a solution to problem (\ref{equ5}).

Let us introduce the reduced energy
\[
\tilde{J}_{\varepsilon}(t,\eta):=J_{\varepsilon}\left(W_{\delta_{\varepsilon}(t),\eta}+\Phi_{\delta_{\varepsilon}(t),\eta}\right)
\]
where $W_{\delta_{\varepsilon}(t),\eta}$ is defined in (\ref{eq:W}),
$\delta_{\varepsilon}(t)=\sqrt{|\varepsilon|t}$ (see \eqref{delta}) and $\Phi_{\delta_{\varepsilon}(t),\eta}$
is given in Proposition \ref{prop:Phi}.

\begin{prop}
\label{lem:4.1}\begin{itemize}
\item[(i)] If $(t ,\eta )$ is a critical
point of $\tilde{J}_{\varepsilon}$, then $W_{\delta_{\varepsilon}(t ),\eta }+\Phi_{\delta_{\varepsilon}(t  ),\eta  }$
is a solution of (\ref{eq:1.6}) and then is a solution of problem
(\ref{equ5}).
\item[(ii)]
We have
\begin{align}\label{jtilde}
\tilde J_{\varepsilon}\left(t,\eta \right)=a(\xi_{0})\left[ a_{m}-b_{m}\varepsilon\log |\varepsilon| +c_{m}\varepsilon + d_m |\varepsilon|\Phi(t,\eta)\right]+o(|\varepsilon|)
\end{align}
$C^{1}-$uniformly with respect to $\eta\in\mathbb{R}^{m}$ and $t\in[\alpha,\beta].$

Here
\begin{align}\label{PHI}
\Phi(t,\eta)= &  \left\{{2(m-1)\over(m-2)(m-4)} \left[  h(\xi_{0}) -\frac{m-2}{4(m-1)}S_{g}(\xi_{0})   +{3 (m-2)\over 2(m-1)}\frac{\Delta_g  a(\xi_0)}{ a(\xi_{0})}\right]   +{1\over2} \frac{D^2_g a(\xi_0)[\eta,\eta]}{a(\xi_{0})}\right\}t\nonumber\\ &-{\varepsilon\over|\varepsilon|} {(m-2)^2\over8}\log t
\end{align}
and $a_m,\dots,d_m$ are  constants which only depend on $m$.  \end{itemize}
 \end{prop}\begin{proof}
 The proof of [i] is quite standard and can be obtained arguing exactly as in the proof of Proposition 2.2 of \cite{MPV}.
 \\

The proof of  [ii] follows in two steps.
 \\
 {\bf Step 1} We prove that
$$
\tilde{J}_{\varepsilon}(t,\eta)=J_{\varepsilon}\left(W_{\delta_{\varepsilon}(t),\eta}\right)+o(|\varepsilon|)
$$
$C^{1}-$uniformly with respect to $\eta\in\mathbb{R}^{m}$ and $t\in[\alpha,\beta]$.\\
\begin{proof}
First, let us prove the $C^ 0-$estimate.
We have that
\begin{align*}
&\tilde{J}_{\varepsilon}(t,\eta)-J_{\varepsilon}\left(W_{\delta_{\varepsilon}(t),\eta}\right)\\ =& \int_{M}a(x)\left(\nabla_{g}W_{\delta_{\varepsilon}(t),\eta}\nabla_{g}\Phi_{\delta_{\varepsilon}(t),\eta}+
hW_{\delta_{\varepsilon}(t),\eta}\Phi_{\delta_{\varepsilon}(t),\eta}-
f_{\varepsilon}(W_{\delta_{\varepsilon}(t),\eta})\Phi_{\delta_{\varepsilon}(t),\eta}\right)d\mu_{g}\\
 & -\int_{M}a(x)\left(\left(W_{\delta_{\varepsilon}(t),\eta}+
 \Phi_{\delta_{\varepsilon}(t),\eta}\right)^{2^{*}-\varepsilon}-
 \left(W_{\delta_{\varepsilon}(t),\eta}\right)^{2^{*}-\varepsilon}
  -f_{\varepsilon}(W_{\delta_{\varepsilon}(t),\eta})\Phi_{\delta_{\varepsilon}(t),\eta}\right)d\mu_{g}\\
 & +\frac{1}{2}\|\Phi_{\delta_{\varepsilon}(t),\eta}\|_{H}^{2}
\end{align*}
Since
\begin{align*}
\int_{M}a(x)\nabla_{g}W_{\delta_{\varepsilon}(t),\eta}\nabla\Phi_{\delta_{\varepsilon}(t),\eta}d\mu_{g}=&
-\int_{M}\nabla_{g}a(x)\nabla_{g}W_{\delta_{\varepsilon}(t),\eta}\Phi_{\delta_{\varepsilon}(t),\eta}d\mu_{g}\\ &-
\int_{M}a(x)\Delta_{g}W_{\delta_{\varepsilon}(t),\eta}\Phi_{\delta_{\varepsilon}(t),\eta}d\mu_{g}
\end{align*}
we need an estimate of the term
\[
\int_{M}\nabla_{g}a(x)\nabla_{g}W_{\delta_{\varepsilon}(t),\eta}\Phi_{\delta_{\varepsilon}(t),\eta}d\mu_{g}.
\]
By (\ref{eq:4.3}) we get that
\[
\|\nabla_{g}a(x)\nabla_{g}W_{\delta_{\varepsilon}(t),\eta}\|_{\frac{2m}{m+2}}=O(|\varepsilon|)
\]
and, by Holder inequality and by Proposition \ref{prop:Phi} we obtain
\[
\left|\int_{M}\nabla_{g}a(x)\nabla_{g}W_{\delta_{\varepsilon}(t),\eta}\Phi_{\delta_{\varepsilon}(t),\eta}d\mu_{g}\right|=o(|\varepsilon|).
\]
The following estimate is analogous to (4.11) in {[}\cite{MPV}, Lemma 4.2{]}:
for any $\theta\in(0,1)$
\[
\|-\Delta_{g}W_{\delta_{\varepsilon}(t),\eta}+hW_{\delta_{\varepsilon}(t),\eta}-f_{\varepsilon}(W_{\delta_{\varepsilon}(t),\eta})\|_{\frac{2m}{m+2}}=o(|\varepsilon|^{\theta}).
\]
Thus
\[
\left|\int_{M}a(x)\left(-\Delta_{g}W_{\delta_{\varepsilon}(t),\eta}+hW_{\delta_{\varepsilon}(t),\eta}-f_{\varepsilon}(W_{\delta_{\varepsilon}(t),\eta})\right)\Phi_{\delta_{\varepsilon}(t),\eta}d\mu_{g}\right|=o(|\varepsilon|).
\]
Similarly, following again {[}\cite{MPV}, Lemma 4.2{]} it is easy to prove
that
\[
\left|\int_{M}a(x)\left(\left(W_{\delta_{\varepsilon}(t),\eta}+\Phi_{\delta_{\varepsilon}(t),\eta}\right)^{2^{*}-\varepsilon}-\left(W_{\delta_{\varepsilon}(t),\eta}\right)^{2^{*}-\varepsilon}-f_{\varepsilon}(W_{\delta_{\varepsilon}(t),\eta})\Phi_{\delta_{\varepsilon}(t),\eta}\right)d\mu_{g}\right|=o(|\varepsilon|).
\]
that concludes the proof of the $C^0-$estimate.

Let us prove the  $C^{1}-$estimate.
We point out that
\begin{equation}
\frac{\partial}{\partial t}W_{\delta_{\varepsilon}(t),\eta}(x)=\frac{1}{2t}Z_{\delta_{\varepsilon}(t),\eta}^{0}(x)\label{eq:2.6}
\end{equation}
\begin{equation}
\frac{\partial}{\partial\eta_{k}}W_{\delta_{\varepsilon}(t),\eta}(x)=Z_{\delta_{\varepsilon}(t),\eta}^{k}(x)\text{ for }1\le k\le m\label{eq:2.7}
\end{equation}
and
\begin{equation}
\|Z_{\delta,\eta}^{k}\|_{H}^{2}\rightarrow a(\xi_{0})\|V_{k}\|_{D^{1,2}(\mathbb{R}^{m})}\text{ for }0\le k\le m.\label{eq:2.8}
\end{equation}

 We have
\[
\frac{\partial}{\partial t}J_{\varepsilon}\left(W_{\delta_{\varepsilon}(t),\eta}+\Phi_{\delta_{\varepsilon}(t),\eta}\right)-\frac{\partial}{\partial t}J_{\varepsilon}\left(W_{\delta_{\varepsilon}(t),\eta}\right)
\]
 and
\[
\frac{\partial}{\partial\eta_{k}}J_{\varepsilon}\left(W_{\delta_{\varepsilon}(t),\eta}+\Phi_{\delta_{\varepsilon}(t),\eta}\right)-\frac{\partial}{\partial\eta_{k}}J_{\varepsilon}\left(W_{\delta_{\varepsilon}(t),\eta}\right)\text{ for }1\le k\le m.
\]
By (\ref{eq:2.6}) and (\ref{eq:2.7}) we have
\begin{align*}
&\frac{\partial}{\partial t}J_{\varepsilon}\left(W_{\delta_{\varepsilon}(t),\eta}+\Phi_{\delta_{\varepsilon}(t),\eta}\right)-
\frac{\partial}{\partial t}J_{\varepsilon}\left(W_{\delta_{\varepsilon}(t),\eta}\right)\\ = &
\left(J_{\varepsilon}'\left(W_{\delta_{\varepsilon}(t),\eta}+\Phi_{\delta_{\varepsilon}(t),\eta}\right)-J_{\varepsilon}'\left(W_{\delta_{\varepsilon}(t),\eta}\right)\right)\left[\frac{\partial}{\partial t}W_{\delta_{\varepsilon}(t),\eta}\right]\\
 & +J_{\varepsilon}'\left(W_{\delta_{\varepsilon}(t),\eta}+\Phi_{\delta_{\varepsilon}(t),\eta}\right)\left[\frac{\partial}{\partial t}\Phi_{\delta_{\varepsilon}(t),\eta}\right]\\
= & \left(J_{\varepsilon}'\left(W_{\delta_{\varepsilon}(t),\eta}+\Phi_{\delta_{\varepsilon}(t),\eta}\right)-J_{\varepsilon}'\left(W_{\delta_{\varepsilon}(t),\eta}\right)\right)\left[\frac{1}{2t}Z_{\delta_{\varepsilon}(t),\eta}^{0}\right]\\
 & +J_{\varepsilon}'\left(W_{\delta_{\varepsilon}(t),\eta}+\Phi_{\delta_{\varepsilon}(t),\eta}\right)\left[\frac{\partial}{\partial t}\Phi_{\delta_{\varepsilon}(t),\eta}\right].
\end{align*}
Analogously
\begin{align*}&
\frac{\partial}{\partial\eta_{k}}J_{\varepsilon}\left(W_{\delta_{\varepsilon}(t),\eta}+
\Phi_{\delta_{\varepsilon}(t),\eta}\right)-\frac{\partial}{\partial\eta_{k}}
J_{\varepsilon}\left(W_{\delta_{\varepsilon}(t),\eta}\right)\\ = & \left(J_{\varepsilon}'\left(W_{\delta_{\varepsilon}(t),\eta}+\Phi_{\delta_{\varepsilon}(t),\eta}\right)-J_{\varepsilon}'\left(W_{\delta_{\varepsilon}(t),\eta}\right)\right)\left[Z_{\delta_{\varepsilon}(t),\eta}^{k}\right]\\
 & +J_{\varepsilon}'\left(W_{\delta_{\varepsilon}(t),\eta}+\Phi_{\delta_{\varepsilon}(t),\eta}\right)\left[\frac{\partial}{\partial\eta_{k}}\Phi_{\delta_{\varepsilon}(t),\eta}\right].
\end{align*}
For $0\le k\le m$ we have
\begin{align*}
&\left(J_{\varepsilon}'\left(W_{\delta_{\varepsilon}(t),\eta}
+\Phi_{\delta_{\varepsilon}(t),\eta}\right)-J_{\varepsilon}'\left(W_{\delta_{\varepsilon}(t),\eta}\right)
\right)\left[Z_{\delta_{\varepsilon}(t),\eta}^{k}\right]\\ =&-\int_{M}\nabla_{g}
a(x)\nabla_{g}Z_{\delta_{\varepsilon}(t),\eta}^{k}\Phi_{\delta_{\varepsilon}(t),\eta}d\mu_{g}\\ &
+\int_{M}a(x)\left(-\Delta_{g}Z_{\delta_{\varepsilon}(t),\eta}^{k}
+hZ_{\delta_{\varepsilon}(t),\eta}^{k}-f_{\varepsilon}'(W_{\delta_{\varepsilon}(t),\eta})
Z_{\delta_{\varepsilon}(t),\eta}^{k}\right)\Phi_{\delta_{\varepsilon}(t),\eta}d\mu_{g}\\ &
+\int_{M}a(x)\left(f_{\varepsilon}(W_{\delta_{\varepsilon}(t),\eta}
+\Phi_{\delta_{\varepsilon}(t),\eta})-f_{\varepsilon}(W_{\delta_{\varepsilon}(t),\eta})
-f_{\varepsilon}'(W_{\delta_{\varepsilon}(t),\eta})\Phi_{\delta_{\varepsilon}(t),\eta}\right)
Z_{\delta_{\varepsilon}(t),\eta}^{k}d\mu_{g}.
\end{align*}
At this point, by (\ref{eq:2.8}), we have
\[
\left|\int_{M}\nabla_{g}a(x)\nabla_{g}Z_{\delta_{\varepsilon}(t),\eta}^{k}\Phi_{\delta_{\varepsilon}(t),\eta}d\mu_{g}\right|\le\|\Phi_{\delta_{\varepsilon}(t),\eta}\|_{L^{2}}\|\nabla_{g}a(x)\nabla_{g}Z_{\delta_{\varepsilon}(t),\eta}^{k}\|_{L^{2}}=o(|\varepsilon|).
\]
Arguing as in (4.26) of {[}\cite{MPV}, Lemma 4.2{]} we have that, for $0\le k\le m$,
and for $\theta\in(0,1)$
\[
\|-\Delta_{g}Z_{\delta_{\varepsilon}(t),\eta}^{k}+hZ_{\delta_{\varepsilon}(t),\eta}^{k}-f_{\varepsilon}'(W_{\delta_{\varepsilon}(t),\eta})Z_{\delta_{\varepsilon}(t),\eta}^{k}\|_{\frac{2m}{m+2}}=O(|\varepsilon|^{\theta}).
\]
This implies
\[
\left|\int_{M}a(x)\left(-\Delta_{g}Z_{\delta_{\varepsilon}(t),\eta}^{k}+hZ_{\delta_{\varepsilon}(t),\eta}^{k}-f_{\varepsilon}'(W_{\delta_{\varepsilon}(t),\eta})Z_{\delta_{\varepsilon}(t),\eta}^{k}\right)\Phi_{\delta_{\varepsilon}(t),\eta}d\mu_{g}\right|=o(|\varepsilon|).
\]
 Again, arguing as in {[}\cite{MPV}, Lemma 4.2{]} we have
\[
\left|\int_{M}a(x)\left(f_{\varepsilon}(W_{\delta_{\varepsilon}(t),\eta}+\Phi_{\delta_{\varepsilon}(t),\eta})-f_{\varepsilon}(W_{\delta_{\varepsilon}(t),\eta})-f_{\varepsilon}'(W_{\delta_{\varepsilon}(t),\eta})\Phi_{\delta_{\varepsilon}(t),\eta}\right)Z_{\delta_{\varepsilon}(t),\eta}^{k}d\mu_{g}\right|=o(|\varepsilon|).
\]
It remains to estimate the term $J_{\varepsilon}'\left(W_{\delta_{\varepsilon}(t),\eta}+\Phi_{\delta_{\varepsilon}(t),\eta}\right)\left[\frac{\partial}{\partial t}\Phi_{\delta_{\varepsilon}(t),\eta}\right]$.
By (\ref{eq:1.7}), since $\Phi_{\delta_{\varepsilon}(t),\eta}\in K_{\delta_{\varepsilon}(t),\eta}^{\bot}$
we have
\begin{align*}
J_{\varepsilon}'\left(W_{\delta_{\varepsilon}(t),\eta}+\Phi_{\delta_{\varepsilon}(t),\eta}\right)\left[\frac{\partial}{\partial t}\Phi_{\delta_{\varepsilon}(t),\eta}\right] & =\sum_{j=0}^{m}\lambda_{\delta_{\varepsilon}(t),\eta}^{j}\left\langle Z_{\delta_{\varepsilon}(t),\eta}^{j},\frac{\partial}{\partial t}\Phi_{\delta_{\varepsilon}(t),\eta}\right\rangle _{H}\\
 & =-\sum_{j=0}^{m}\lambda_{\delta_{\varepsilon}(t),\eta}^{j}\left\langle \frac{\partial}{\partial t}Z_{\delta_{\varepsilon}(t),\eta}^{j},\Phi_{\delta_{\varepsilon}(t),\eta}\right\rangle _{H}
\end{align*}
By easy computation we have that $\left\Vert \frac{\partial}{\partial t}Z_{\delta_{\varepsilon}(t),\eta}^{j}\right\Vert _{H}=O(1)$
and $\left\Vert \frac{\partial}{\partial\eta_{k}}Z_{\delta_{\varepsilon}(t),\eta}^{j}\right\Vert _{H}=O(1)$
for all $1\le k\le m$. By {[}\cite{MPV}, Lemma 4.2{]} we have that $\sum_{j=0}^{m}\left|\lambda_{\delta_{\varepsilon}(t),\eta}^{j}\right|=O(|\varepsilon|^{1/2}).$
This, in light of Proposition \ref{prop:Phi} ensures that
\[
J_{\varepsilon}'\left(W_{\delta_{\varepsilon}(t),\eta}+\Phi_{\delta_{\varepsilon}(t),\eta}\right)\left[\frac{\partial}{\partial t}\Phi_{\delta_{\varepsilon}(t),\eta}\right]=o(|\varepsilon|).
\]
The estimate for $J_{\varepsilon}'\left(W_{\delta_{\varepsilon}(t),\eta}+\Phi_{\delta_{\varepsilon}(t),\eta}\right)\left[\frac{\partial}{\partial\eta_{k}}\Phi_{\delta_{\varepsilon}(t),\eta}\right]$,
$k\in[1,m]$ can be obtained in a similar way. This concludes the
proof. \end{proof}

 {\bf Step 2} We prove that $J_{\varepsilon}\left(W_{\delta_{\varepsilon}(t),\eta}\right)$ satisfies expansion \eqref{jtilde}
 $C^{1}-$uniformly with respect to $\eta\in\mathbb{R}^{m}$ and $t\in[\alpha,\beta]$.

\begin{proof} Let us prove the $C^{0}-$estimate.

It holds
\begin{align*}
 J_{\varepsilon}\left(W_{\delta_{\varepsilon}(t),\eta}\right)= &  a(\xi_{0})\left( \underbrace{\frac{1}{2}\int_{M} |\nabla_{g}W_{\delta,\eta}|^{2}d\mu_{g} }_{I_1}+\underbrace{\frac{1}{2}\int_{M} h(x)W_{\delta,\eta}^{2} d\mu_{g}}_{I_2} -\underbrace{ \frac{1}{2_{m}^{*}-\varepsilon}\int_{M}W_{\delta,\eta}^{2_{m}^{*}-\varepsilon}d\mu_{g}}_{I_3}\right) \\
 & +\underbrace{\frac{1}{2}\int_{M}(a(x)-a(\xi_{0}))  |\nabla_{g}W_{\delta,\eta}|^{2}d\mu_{g}}_{I_4}+\underbrace{\frac{1}{2}\int_{M}(a(x)-a(\xi_{0})) h(x)W_{\delta,\eta}^{2}d\mu_{g}}_{I_5}\\ &-\underbrace{\frac{1}{2_{m}^{*}-\varepsilon}\int_{M}(a(x)-a(\xi_{0})) W_{\delta,\eta}^{2_{m}^{*}-\varepsilon} d\mu_{g}}_{I_6}.
\end{align*}

First of all, let us estimate  the integrals $I_1,$ $I_2$ and $I_3.$
\begin{align*}
I_{1}= & \frac{\delta^{-m+2}}{2}\int_{\mathbb{R}^{m}}g^{ij}(y)\frac{\partial}{\partial y_{i}}\left(\chi(y-\delta\eta)U\left(\frac{y}{\delta}-\eta\right)\right)\frac{\partial}{\partial z_{j}}\left(\chi(y-\delta\eta)U\left(\frac{y}{\delta}-\eta\right)\right)|g(y)|^{1/2}dy\\
= & \frac{1}{2}\int_{\mathbb{R}^{m}}g^{ij}(\delta(z+\eta))\frac{\partial}{\partial z_{i}}\left(\chi(\delta z)U(z)\right)\frac{\partial}{\partial z_{j}}\left(\chi(\delta z)U(z)\right)|g(\delta(z+\eta))|^{1/2}dz\\
= & \frac{1}{2}\int_{\mathbb{R}^{m}}g^{ij}(\delta(z+\eta))\frac{\partial U}{\partial z_{i}}(z)\frac{\partial U}{\partial z_{j}}(z)|g(\delta(z+\eta))|^{1/2}dz+o(\delta^{2})\\
= & \frac{1}{2}\int_{\mathbb{R}^{m}}\left(\delta_{ij}+\frac{\delta^{2}}{2}\sum_{i,j,r,k=1}^{n}\frac{\partial^{2}g^{ij}}{\partial y_{r}\partial y_{k}}(0)(z_{r}+\eta_{r})(z_{k}+\eta_{k})\right)\frac{\partial U}{\partial z_{i}}(z)\frac{\partial U}{\partial z_{j}}(z)\times\\
 & \times\left(1-\frac{\delta^{2}}{4}\sum_{s,r,k=1}^{n}\frac{\partial^{2}g^{ss}}{\partial y_{r}\partial y_{k}}(0)(z_{r}+\eta_{r})(z_{k}+\eta_{k})\right)dz+o(\delta^{2})\\
 = & \frac{1}{2}\int_{\mathbb{R}^{m}}|\nabla U|^{2}dz+\frac{|\varepsilon|t}{4}\sum_{i,j,r,k=1}^{m}\frac{\partial^{2}g^{ij}}{\partial y_{r}\partial y_{k}}(0)\int_{\mathbb{R}^{m}}\frac{\partial U}{\partial z_{i}}(z)\frac{\partial U}{\partial z_{j}}(z)z_{r}z_{k}dz\\
 & +\frac{|\varepsilon|t}{4}\sum_{i,j,r,k=1}^{m}\frac{\partial^{2}g^{ij}}{\partial y_{r}\partial y_{k}}(0)\eta_{r}\eta_{k}\int_{\mathbb{R}^{m}}\frac{\partial U}{\partial z_{i}}(z)\frac{\partial U}{\partial z_{j}}(z)dz\\
 & -\frac{|\varepsilon|t}{8}\sum_{s,k=1}^{m}\frac{\partial^{2}g^{ss}}{\partial y_{k}^{2}}(0)\int_{\mathbb{R}^{m}}|\nabla U|^{2}z_{k}^{2}dz\\
 & -\frac{|\varepsilon|t}{8}\sum_{s,r,k=1}^{m}\frac{\partial^{2}g^{ss}}{\partial y_{r}\partial y_{k}}(0)\eta_{r}\eta_{k}\int_{\mathbb{R}^{m}}|\nabla U|^{2}dzdz+o(\varepsilon)\\\
 = & \frac{1}{2}\int_{\mathbb{R}^{m}}|\nabla U|^{2}dz+\frac{|\varepsilon|t}{4}\sum_{i,j,r,k=1}^{m}\frac{\partial^{2}g^{ij}}{\partial y_{r}\partial y_{k}}(0)\int_{\mathbb{R}^{m}}\left(\frac{U'(z)}{|z|}\right)^{2}z_{i}z_{j}z_{r}z_{k}dz\nonumber \\
 & +\frac{|\varepsilon|t}{4}\sum_{i,r,k=1}^{m}\frac{\partial^{2}g^{ii}}{\partial y_{r}\partial y_{k}}(0)\eta_{r}\eta_{k}\int_{\mathbb{R}^{m}}\left(\frac{U'(z)}{|z|}\right)^{2}z_{i}^{2}\nonumber \\
 & -\frac{|\varepsilon|t}{8}\sum_{i,k=1}^{m}\frac{\partial^{2}g^{ii}}{\partial y_{k}^{2}}(0)\int_{\mathbb{R}^{m}}\left(U'(z)\right)^{2}z_{k}^{2}dz\nonumber \\
 & -\frac{|\varepsilon|t}{8}\sum_{i,r,k=1}^{m}\frac{\partial^{2}g^{ii}}{\partial y_{r}\partial y_{k}}(0)\eta_{r}\eta_{k}\int_{\mathbb{R}^{m}}\left(U'(z)\right)^{2}dz+o(\varepsilon).
\end{align*}

We set $\tilde{h}(y)=h(\exp_{\xi_0}(y))$. Then we get

 \begin{align*}
I_{2}= & \frac{\delta^{-m+2}}{2}\int_{\mathbb{R}^{m}}\tilde{h}(y)\chi^{2}(y-\delta\eta)U^{2}\left(\frac{y}{\delta}-\eta\right)|g(y)|^{1/2}dy\\
= & \frac{\delta^{2}}{2}\int_{\mathbb{R}^{m}}\tilde{h}(\delta(z+\eta))U^{2}(z)|g(\delta(z+\eta))|^{1/2}dz+o(\delta^{2})\\
= & \frac{\delta^{2}}{2}\int_{\mathbb{R}^{m}}\left(\tilde{h}(0)+O(\delta)\right)U^{2}(z)(1+O(\delta^{2}))dz+o(\delta^{2})\\
= & \frac{|\varepsilon|t}{2}\tilde{h}(0)\int_{\mathbb{R}^{m}}U^{2}(z)dz+o(\varepsilon).
\end{align*}

We notice that, by direct computation, and considering that $\delta=\sqrt{|\varepsilon|t}$
\begin{align*}
\frac{1}{2_{m}^{*}-\varepsilon} & =\frac{1}{2_{m}^{*}}+\frac{\varepsilon}{\left(2_{m}^{*}\right)^{2}}+o(\varepsilon);\\
\delta^{\varepsilon\frac{m-2}{2}} & =1+\frac{m-2}{4}\varepsilon\log\left(|\varepsilon|t\right)+o(\varepsilon);\\
U^{2_{m}^{*}-\varepsilon} & =U^{2_{m}^{*}}-\varepsilon U^{2_{m}^{*}}\log U+o(\varepsilon).
\end{align*}

Therefore

\begin{align*}
I_{3}= & \frac{\delta^{-\frac{m-2}{2}(2_{m}^{*}-\varepsilon)}}{2_{m}^{*}-\varepsilon}\int_{\mathbb{R}^{m}}\chi^{2_{m}^{*}-\varepsilon}(y-\delta\eta)U^{2_{m}^{*}-\varepsilon}\left(\frac{y}{\delta}-\eta\right)|g(y)|^{1/2}dy\\
= & \frac{\delta^{\varepsilon\frac{m-2}{2}}}{2_{m}^{*}-\varepsilon}\int_{\mathbb{R}^{m}}U^{2_{m}^{*}-\varepsilon}(z)|g(\delta(z+\eta))|^{1/2}dz+o(\delta^{2})\\
= & \frac{\delta^{\varepsilon\frac{m-2}{2}}}{2_{m}^{*}-\varepsilon}\int_{\mathbb{R}^{m}}U^{2_{m}^{*}-\varepsilon}(z)\left(1-\frac{\delta^{2}}{4}\sum_{i,r,k=1}^{n}\frac{\partial^{2}g^{ii}}{\partial y_{r}\partial y_{k}}(0)(z_{r}+\eta_{r})(z_{k}+\eta_{k})\right)dz+o(\delta^{2})\\
= & \frac{\delta^{\varepsilon\frac{m-2}{2}}}{2_{m}^{*}-\varepsilon}\int_{\mathbb{R}^{m}}U^{2_{m}^{*}-\varepsilon}(z)\left(1-\frac{\delta^{2}}{4}\sum_{i,r,k=1}^{n}\frac{\partial^{2}g^{ii}}{\partial y_{r}\partial y_{k}}(0)(z_{r}z_{k}+\eta_{r}\eta_{k})\right)dz+o(\delta^{2})\\ = & \left(\frac{1}{2_{m}^{*}}+\frac{\varepsilon}{\left(2_{m}^{*}\right)^{2}}\right)\left(1+\frac{m-2}{4}\varepsilon\log\left(|\varepsilon|t\right)\right)\int_{\mathbb{R}^{m}}\left(U^{2_{m}^{*}}(z)-\varepsilon U^{2_{m}^{*}}(z)\log U(z)\right)\\
 & \times\left(1-\frac{|\varepsilon|t}{4}\sum_{i,r,k=1}^{n}\frac{\partial^{2}g^{ii}}{\partial y_{r}\partial y_{k}}(0)(z_{r}z_{k}+\eta_{r}\eta_{k})\right)dz+o(\varepsilon)\\
= & \frac{1}{2_{m}^{*}}\int_{\mathbb{R}^{m}}U^{2_{m}^{*}}+\frac{\varepsilon}{2_{m}^{*}}\left(\frac{1}{2_{m}^{*}}\int_{\mathbb{R}^{m}}U^{2_{m}^{*}}(z)dz-\int_{\mathbb{R}^{m}}U^{2_{m}^{*}}(z)\log U(z)dz\right)\\
 & +\frac{1}{2_{m}^{*}}\frac{m-2}{4}\varepsilon\log\left(|\varepsilon|t\right)\int_{\mathbb{R}^{m}}U^{2_{m}^{*}}(z)dz+\\
 & -\frac{1}{2_{m}^{*}}\frac{|\varepsilon|t}{4}\sum_{i,r,k=1}^{m}\frac{\partial^{2}g^{ii}}{\partial y_{r}\partial y_{k}}(0)\eta_{r}\eta_{k}\int_{\mathbb{R}^{m}}U^{2_{m}^{*}}(z)dz\\
 & -\frac{1}{2_{m}^{*}}\frac{|\varepsilon|t}{4}\sum_{i,k=1}^{m}\frac{\partial^{2}g^{ii}}{\partial y_{k}^{2}}(0)\int_{\mathbb{R}^{m}}U^{2_{m}^{*}}(z)z_{r}^{2}dz+o(\varepsilon)
\end{align*}

Therefore we get
\begin{align*}
&I_1+I_2-I_3=   \frac{1}{2}\int_{\mathbb{R}^{m}}|\nabla U|^{2}dz--
 \frac{1}{2_{m}^{*}}\int_{\mathbb{R}^{m}}U^{2_{m}^{*}}dz\\
 &
  - \frac{1}{2_{m}^{*}}\frac{m-2}{4}\varepsilon\log\left(|\varepsilon|t\right)\int_{\mathbb{R}^{m}}U^{2_{m}^{*}}(z)dz+\\
  & -\frac{\varepsilon}{2_{m}^{*}}\left(\frac{1}{2_{m}^{*}}\int_{\mathbb{R}^{m}}U^{2_{m}^{*}}(z)dz-\int_{\mathbb{R}^{m}}U^{2_{m}^{*}}(z)\log U(z)dz\right)\\
   &+ \frac{|\varepsilon|t}{2}\tilde{h}(0)\int_{\mathbb{R}^{m}}U^{2}(z)dz
\\
&+ {|\varepsilon|t}\(\frac{1}{4}\sum_{i,j,r,k=1}^{m}\frac{\partial^{2}g^{ij}}{\partial y_{r}\partial y_{k}}(0)\int_{\mathbb{R}^{m}}\left(\frac{U'(z)}{|z|}\right)^{2}z_{i}z_{j}z_{r}z_{k}dz \right.\\ &\left.\qquad\quad -\sum_{i,k=1}^{m}\frac{\partial^{2}g^{ii}}{\partial y_{k}^{2}}(0)\(\frac{1}{8}\int_{\mathbb{R}^{m}}\left(U'(z)\right)^{2}z_{k}^{2}dz  +\frac{m-2}{8m}   \int_{\mathbb{R}^{m}}U^{2_{m}^{*}}(z)z_{r}^{2}dz\)\)\\
  &+{|\varepsilon|t}\underbrace{\sum_{i,r,k=1}^{m}\frac{\partial^{2}g^{ii}}{\partial y_{r}\partial y_{k}}(0)\eta_{r}\eta_{k}\(\frac{m-2}{8m}  \int_{\mathbb{R}^{m}}U^{2_{m}^{*}}(z)dz
  +\frac{1}{4}\int_{\mathbb{R}^{m}}\left(\frac{U'(z)}{|z|}\right)^{2}z_{i}^{2}dz -\frac{1}{8}\int_{\mathbb{R}^{m}}\left(U'(z)\right)^{2}dz\)}_{=0}
  \\ &+o(\varepsilon)\\
  &={K_m^{-m}\over m}\left[1-{(m-2)^2\over8}\varepsilon\log(|\varepsilon|t)-c_m\varepsilon+{2(m-1)\over (m-2)(m-4)}\(h(\xi)-{(m-2)\over 4(m-1)}S_g(\xi)\){|\varepsilon|t}+o(\varepsilon)\right],\end{align*}
 where $c_m$ is a constant which only depends on $m, $   $K_m:=\sqrt{4\over m(m-2)\omega_m^{2/m}}$ and $\omega_m$ is the volume of the unit $m-$sphere.
(see also    [\cite{MPV}, Lemma 4.1]).
Here we used the following facts
 \begin{align*}&{\displaystyle \sum_{i,j=1}^{m}\frac{\partial^{2}g^{ii}}{\partial z_{j}^{2}}(0)-\sum_{i,j=1}^{m}\frac{\partial^{2}g^{ij}}{\partial z_{i}\partial z_{j}}(0)=S_{g}(\xi_{0})},\\
 &\frac{1}{2}\int_{\mathbb{R}^{m}}|\nabla U|^{2}z_{i}^{2}dz-\frac{1}{2^{*}}\int_{\mathbb{R}^{m}}U^{2^{*}}z_{i}^{2}dz=\int_{\mathbb{R}^{m}}\left(\frac{\partial U}{\partial z_{i}}\right)^{2}z_{i}^{2}dz\\
 &\int_{\mathbb{R}^{m}}\left(\frac{U'(z)}{|z|}\right)^{2}z_{i}^{4}dz=3\int_{\mathbb{R}^{m}}\left(\frac{U'(z)}{|z|}\right)^{2}z_{i}^{2}z_{j}^{2}dz\\ &\int_{\mathbb{R}^{m}}\left(\frac{U'(z)}{|z|}\right)^{2}z_{i}z_{j}z_{k}z_{h}dz=\frac{1}{2}\int_{\mathbb{R}^{m}}\left(\frac{U'(z)}{|z|}\right)^{2}z_{1}^{4}dz(\delta_{ij}\delta_{hk}+\delta_{ik}\delta_{jh}+\delta_{ih}\delta_{jk}) . \end{align*}

 Now, let us estimate  the integrals $I_4,$ $I_5$ and $I_6.$
 We set $\tilde{a}(y)=a(\exp_{\xi_0}(y))$ and we denote
  by $\frac{\partial\tilde{a}}{\partial y_{s}}$  the derivative
of $\tilde{a}$ with respect to its $s$-th variable. Therefore, we have

\begin{align}
I_{4}&=\frac{1}{2}\int_{\mathbb{R}^{m}}[\sum_{i,j}\tilde{a}(\delta(z+\eta))-\tilde{a}(0)]g^{ij}(\delta(z+\eta))\frac{\partial}{\partial z_{i}}\left(\chi(\delta z)U(z)\right)\frac{\partial}{\partial z_{j}}\left(\chi(\delta z)U(z)\right)|g(\delta(z+\eta))|^{1/2}dz\nonumber\\ &
=\frac{\delta^{2}}{4}\int_{\mathbb{R}^{m}}\sum_{s,k}\frac{\partial^{2}\tilde{a}}{\partial y_{s}\partial y_{k}}(0)(z_{s}+\eta_{s})(z_{k}+\eta_{k})|\nabla U|^{2}dz+o(\delta^{2})\nonumber\\ &
=\sum_{s,k}\frac{\delta^{2}}{4}\frac{\partial^{2}\tilde{a}}{\partial y_{s}\partial y_{k}}(0)\int_{\mathbb{R}^{m}}z_{s}z_{k}|\nabla U|^{2}dz+\sum_{s,k}\frac{\delta^{2}}{4}\frac{\partial^{2}\tilde{a}}{\partial y_{s}\partial y_{k}}(0)\eta_{s}\eta_{k}\int_{\mathbb{R}^{m}}|\nabla U|^{2}dz+o(\delta^{2})\nonumber\\ &
=|\varepsilon|t\left\{ \frac{1}{4}\sum_k\frac{\partial^{2}\tilde{a}}{\partial^{2}y_{k}}(0)\int_{\mathbb{R}^{m}}z_{k}^{2}|\nabla U|^{2}dz+\sum_{s,k}\frac{\partial^{2}\tilde{a}}{\partial y_{s}\partial y_{k}}(0)\frac{\eta_{s}\eta_{k}}{4}\int_{\mathbb{R}^{m}}|\nabla U|^{2}dz\right\} +o(\varepsilon)\nonumber\\ &
=|\varepsilon|t\left\{ \frac{1}{4m}\sum\limits_{k=1}^m\frac{\partial^{2}\tilde{a}}{\partial^{2}y_{k}}(0)\int_{\mathbb{R}^{m}}|z|^{2}|\nabla U|^{2}dz+{1\over4}\sum\limits_{k,s=1}^m\frac{\partial^{2}\tilde{a}}{\partial y_{s}\partial y_{k}}(0) {\eta_{s}\eta_{k}} \int_{\mathbb{R}^{m}}|\nabla U|^{2}dz\right\} +o(\varepsilon).\label{eq:I1esp}
\end{align}
and by mean value theorem we get for some $\theta\in(0,1)$
\begin{align*}
I_{5}&=\frac{\delta^{2}}{2}\int_{\mathbb{R}^{m}}[\tilde{a}(\delta(z+\eta))-\tilde{a}(0)]\tilde{h}(\delta(z+\eta))\chi^{2}(\delta z)U^{2}(z)|g(\delta(z+\eta))|^{1/2}dz\\ &
=\frac{\delta^{4}}{2}\int_{\mathbb{R}^{m}}\sum_{s,k}\frac{\partial^{2}\tilde{a}}{\partial y_{s}\partial y_{k}}(\delta\theta z)\cdot(z_{s}+\eta_{s})(z_{k}+\eta_{k})\tilde{h}(\delta(z+\eta))\chi^{2}(\delta z)U^{2}(z)dz\\ &=O(\delta^{3})=o(\varepsilon).
\end{align*}
  Finally, using that
\[
\left|\int_{\mathbb{R}^{m}}|z|^{2}\left[U^{2_{m}^{*}-\varepsilon}(z)-U^{2_{m}^{*}}(z)\right]dz\right|\le|\varepsilon|\int_{\mathbb{R}^{m}}|z|^{2}U^{2_{m}^{*}-\theta\varepsilon}(z)\ln U(z)dz=o(1).
\]
and also that
\[
\frac{\delta^{\varepsilon\frac{m-2}{2}}}{2_{m}^{*}-\varepsilon}=\frac{1}{2_{m}^{*}-\varepsilon}+o(1)=\frac{1}{2_{m}^{*}}+o(1),
\]
we get
\begin{align*}
I_{6}&=\frac{\delta^{\varepsilon\frac{m-2}{2}}}{2_{m}^{*}-\varepsilon}\int_{\mathbb{R}^{m}}[\tilde{a}(\delta(z+\eta))-\tilde{a}(0)]\chi^{2_{m}^{*}-\varepsilon}(\delta z)U^{2_{m}^{*}-\varepsilon}(z)|g(\delta(z+\eta))|^{1/2}dz\\ &
=\frac{\delta^{2}}{2}\frac{\delta^{\varepsilon\frac{m-2}{2}}}{2_{m}^{*}-\varepsilon}\int_{\mathbb{R}^{m}}\sum_{s,k}\frac{\partial^{2}\tilde{a}}{\partial y_{s}\partial y_{k}}(0)(z_{s}+\eta_{s})(z_{k}+\eta_{k})U^{2_{m}^{*}-\varepsilon}(z)dz+o(\delta^{2})\\ &
=\frac{\delta^{2}}{2}\frac{\delta^{\varepsilon\frac{m-2}{2}}}{2_{m}^{*}-\varepsilon}\sum_{s,k}\frac{\partial^{2}\tilde{a}}{\partial y_{s}\partial y_{k}}(0)\eta_{s}\eta_{k}\int_{\mathbb{R}^{m}}U^{2_{m}^{*}-\varepsilon}(z)dz\\ &
+\frac{\delta^{2}}{2}\frac{\delta^{\varepsilon\frac{m-2}{2}}}{2_{m}^{*}-\varepsilon}\sum_{ k}\frac{\partial^{2}\tilde{a}}{\partial y_{k}^{2}}(0)\int_{\mathbb{R}^{m}}z_{k}^{2}U^{2_{m}^{*}}(z)dz+o(\delta^{2}) \\
&=\frac{|\varepsilon|t}{2\cdot2_{m}^{*}}\left(\sum_{k}\frac{\partial^{2}\tilde{a}}{\partial y_{k}^{2}}(0)\int_{\mathbb{R}^{m}}z_{k}^{2}U^{2_{m}^{*}}(z)dz+\sum_{s,k}\frac{\partial^{2}\tilde{a}}{\partial y_{s}\partial y_{k}}(0)\eta_{s}\eta_{k}\int_{\mathbb{R}^{m}}U^{2_{m}^{*}}(z)dz\right)+o(|\varepsilon|) \\
&= {|\varepsilon|t} \left({m-2\over 4m^2}\sum\limits_{k=1}^m\frac{\partial^{2}\tilde{a}}{\partial y_{k}^{2}}(0)\int_{\mathbb{R}^{m}}|z|^{2}U^{2_{m}^{*}}(z)dz+{m-2\over4m}\sum\limits_{k,s=1}^m\frac{\partial^{2}\tilde{a}}{\partial y_{s}\partial y_{k}}(0) {\eta_{s}\eta_{k}}\int_{\mathbb{R}^{m}}U^{2_{m}^{*}}(z)dz\right)+o(|\varepsilon|).
\end{align*}
Therefore, we get
\begin{align*}I_4 +I_5-I_6=&|\varepsilon|t\left\{\sum\limits_{k=1}^m\frac{\partial^{2}\tilde{a}}{\partial y_{k}^{2}}(0)\left[\frac{1}{4m}\int_{\mathbb{R}^{m}}|z|^{2}|\nabla U|^{2}dz-{m-2\over 4m^2}\int_{\mathbb{R}^{m}}|z|^{2}U^{2_{m}^{*}}(z)dz\right]
\right.\\ &\left.
+\sum\limits_{k,s=1}^m\frac{\partial^{2}\tilde{a}}{\partial y_{s}\partial y_{k}}(0) {\eta_{s}\eta_{k}}\left[\frac{1}{4}\int_{\mathbb{R}^{m}} |\nabla U|^{2}dz-{m-2\over 4m }\int_{\mathbb{R}^{m}}U^{2_{m}^{*}}(z)dz\right]\right\} +o(|\varepsilon|).\end{align*}

A straightforward computation shows that
$$\frac{1}{4m}\int_{\mathbb{R}^{m}}|z|^{2}|\nabla U|^{2}dz-{m-2\over 4m^2}\int_{\mathbb{R}^{m}}|z|^{2}U^{2_{m}^{*}}(z)dz={3K_m^{-m}\over m(m-4)} $$
and
$$\frac{1}{4}\int_{\mathbb{R}^{m}} |\nabla U|^{2}dz-{m-2\over 4m }\int_{\mathbb{R}^{m}}U^{2_{m}^{*}}(z)dz={ K_m^{-m}\over 2m }.$$
It is enough to use the following ingredients.
For any positive real numbers $p$ and $q$ such that $p-q>1$, we let
$$I^q_p=\int_0^{+\infty}\frac{r^q}{\(1+r\)^p}dr.$$
 In particular, there hold
$$I^q_{p+1}=\frac{p-q-1}{p}I^q_p\quad\text{and}\quad I^{q+1}_{p+1}=\frac{q+1}{p-q-1}I^q_{p+1}\,.
$$
Moreover, we have
$$
I_m^{\frac{m}{2}}= \frac{2K_m^{-m}}{\alpha_m^2(m-2)^2\omega_{m-1}}\,.
$$

Finally, we collect all the previous estimates and  we get the   $C^{0}-$estimate,  taking into account that
$$\Delta_g a (\xi_0)=\sum_{k=1}^{m}\frac{\partial^{2}\tilde{a}}{\partial y_{k}^{2}}(0)\ \hbox{and}\ D^2_g a(\xi_0)[\eta,\eta]=\sum_{ k,s=1}^{m} \frac{\partial^{2}\tilde{a}}{\partial y_{k}\partial y_{s}}(0)\eta_{k}\eta_{s} .$$

  Let us prove the $C^{1}-$ estimate.

We first consider the derivative of the term $A(\delta_{\varepsilon}(t),\eta)=\frac{1}{2}\int_{M}a(x)|\nabla_{g}W_{\delta_{\varepsilon}(t),\eta}|^{2}d\mu_{g}$
with respect to $t$. Since $\delta_{\varepsilon}^{'}(t)\delta_{\varepsilon}(t)=\frac{\varepsilon}{2}$
we get
\begin{align*}
\frac{\partial}{\partial t}A(\delta_{\varepsilon}(t),\eta)= & \delta_{\varepsilon}^{'}(t)\frac{\partial}{\partial\delta}A(\delta_{\varepsilon}(t),\eta)=\\
= & \frac{\delta_{\varepsilon}^{'}(t)}{2}\frac{\partial}{\partial\delta}\int_{\mathbb{R}^{m}}\tilde{a}(\delta(z+\eta))g^{ij}(\delta(z+\eta))\frac{\partial}{\partial z_{i}}\left(\chi(\delta z)U(z)\right)\frac{\partial}{\partial z_{j}}\left(\chi(\delta z)U(z)\right)|g(\delta(z+\eta))|^{1/2}dz\\
= & \frac{\delta_{\varepsilon}^{'}(t)}{2}\int_{\mathbb{R}^{m}}\frac{\partial\tilde{a}}{\partial y_{k}}(\delta(z+\eta))(z_{k}+\eta_{k})g^{ij}(\delta(z+\eta))\frac{\partial U(z)}{\partial z_{i}}\frac{\partial U(z)}{\partial z_{j}}|g(\delta(z+\eta))|^{1/2}dz\\
 & +\frac{\delta_{\varepsilon}^{'}(t)}{2}\int_{\mathbb{R}^{m}}\tilde{a}(\delta(z+\eta))\frac{\partial^{2}g^{ij}}{\partial y_{k}}(\delta(z+\eta))(z_{k}+\eta_{k})\frac{\partial U(z)}{\partial z_{i}}\frac{\partial U(z)}{\partial z_{j}}dz\\
 & +\frac{\delta_{\varepsilon}^{'}(t)}{2}\int_{\mathbb{R}^{m}}\tilde{a}(\delta(z+\eta))g^{ij}(\delta(z+\eta))\frac{\partial U(z)}{\partial z_{i}}\frac{\partial U(z)}{\partial z_{j}}\frac{\partial|g|^{1/2}}{\partial y_{k}}(\delta(z+\eta))(z_{k}+\eta_{k})dz+o(\varepsilon)\\
= & \frac{\delta_{\varepsilon}^{'}(t)\delta}{2}\int_{\mathbb{R}^{m}}\frac{\partial^{2}\tilde{a}}{\partial y_{k}\partial y_{r}}(0)(z_{r}+\eta_{r})(z_{k}+\eta_{k})|\nabla U(z)|^2dz\\
 & +\frac{\delta_{\varepsilon}^{'}(t)\delta}{2}\int_{\mathbb{R}^{m}}\tilde{a}(0)\frac{\partial^{2}g^{ij}}{\partial y_{k}\partial y_{r}}(0)(z_{r}+\eta_{r})(z_{k}+\eta_{k})\frac{\partial U(z)}{\partial z_{i}}\frac{\partial U(z)}{\partial z_{j}}dz\\
 & -\frac{\delta_{\varepsilon}^{'}(t)\delta}{4}\int_{\mathbb{R}^{m}}\tilde{a}(0)|\nabla U(z)|^2\frac{\partial^{2}g^{ss}}{\partial y_{k}\partial y_{r}}(0)(z_{r}+\eta_{r})(z_{k}+\eta_{k})dz+o(\varepsilon)\\
= & \frac{\varepsilon}{4}\frac{\partial^{2}\tilde{a}}{\partial y_{k}\partial y_{r}}(0)\eta_{k}\eta_{r}\int_{\mathbb{R}^{m}}|\nabla U(z)|^2dz+\frac{\varepsilon}{4}\frac{\partial^{2}\tilde{a}}{\partial z_{k}^{2}}(0)\int_{\mathbb{R}^{m}}|\nabla U(z)|^2z_{k}^{2}dz\\
 & +\frac{\varepsilon}{4}\tilde{a}(0)\eta_{k}\eta_{r}\frac{\partial^{2}g^{ii}}{\partial y_{k}\partial y_{r}}(0)\int_{\mathbb{R}^{m}}\left(\frac{U'(z)}{|z|}\right)^{2}z_{i}^{2}dz\\
 & +\frac{\varepsilon}{4}\tilde{a}(0)\frac{\partial^{2}g^{ij}}{\partial y_{k}\partial y_{r}}(0)\int_{\mathbb{R}^{m}}\left(\frac{U'(z)}{|z|}\right)^{2}z_{i}z_{j}z_{k}z_{r}dz\\
 & -\frac{\varepsilon}{8}\tilde{a}(0)\eta_{k}\eta_{r}\frac{\partial^{2}g^{ss}}{\partial y_{k}\partial y_{r}}(0)\int_{\mathbb{R}^{m}}|\nabla U(z)|^2dz-\frac{\varepsilon}{8}\tilde{a}(0)\frac{\partial^{2}g^{ss}}{\partial y_{k}^{2}}(0)\int_{\mathbb{R}^{m}}|\nabla U(z)|^2z_{k}^{2}dz+o(\varepsilon)
\end{align*}
Here we  recognize the derivative of $I_1$ with respect to $t.$\\

In a similar way we consider $\frac{\partial}{\partial\eta_{k}}A(\delta_{\varepsilon}(t),\eta)$
for $1\le k\le m$. We have
\begin{align*}
\frac{\partial}{\partial\eta_{k}}A(\delta_{\varepsilon}(t),\eta)= & \frac{1}{2}\frac{\partial}{\partial\partial\eta_{k}}\int_{\mathbb{R}^{m}}\tilde{a}(\delta(z+\eta))g^{ij}(\delta(z+\eta))\frac{\partial}{\partial z_{i}}\left(\chi(\delta z)U(z)\right)\frac{\partial}{\partial z_{j}}\left(\chi(\delta z)U(z)\right)|g(\delta(z+\eta))|^{1/2}dz\\
= & \frac{\delta}{2}\int_{\mathbb{R}^{m}}\frac{\partial\tilde{a}}{\partial y_{k}}(\delta(z+\eta))g^{ij}(\delta(z+\eta))\frac{\partial U(z)}{\partial z_{i}}\frac{\partial U(z)}{\partial z_{j}}|g(\delta(z+\eta))|^{1/2}dz\\
 & +\frac{\delta}{2}\int_{\mathbb{R}^{m}}\tilde{a}(\delta(z+\eta))\frac{\partial^{2}g^{ij}}{\partial y_{k}}(\delta(z+\eta))\frac{\partial U(z)}{\partial z_{i}}\frac{\partial U(z)}{\partial z_{j}}dz\\
 & +\frac{\delta}{2}\int_{\mathbb{R}^{m}}\tilde{a}(\delta(z+\eta))(\delta(z+\eta))\frac{\partial U(z)}{\partial z_{i}}\frac{\partial U(z)}{\partial z_{j}}\frac{\partial|g|^{1/2}}{\partial y_{k}}(\delta(z+\eta))dz+o(\varepsilon)\\
= & \frac{\delta^{2}}{2}\int_{\mathbb{R}^{m}}\frac{\partial^{2}\tilde{a}}{\partial y_{k}\partial y_{r}}(0)(z_{r}+\eta_{r})|\nabla U(z)|^2dz\\
 & +\frac{\delta^{2}}{2}\int_{\mathbb{R}^{m}}\tilde{a}(0)\frac{\partial^{2}g^{ij}}{\partial y_{k}\partial y_{r}}(0)(z_{r}+\eta_{r})\frac{\partial U(z)}{\partial z_{i}}\frac{\partial U(z)}{\partial z_{j}}dz\\
 & -\frac{\delta^{2}}{4}\int_{\mathbb{R}^{m}}\tilde{a}(0)|\nabla U(z)|^2\frac{\partial^{2}g^{ss}}{\partial y_{k}\partial y_{r}}(0)(z_{r}+\eta_{r})dz+o(\varepsilon)\\
= & \frac{\varepsilon}{2}\frac{\partial^{2}\tilde{a}}{\partial y_{k}\partial y_{r}}(0)\eta_{r}\int_{\mathbb{R}^{m}}|\nabla U(z)|^2dz+\frac{\varepsilon}{2}\tilde{a}(0)\eta_{r}\frac{\partial^{2}g^{ii}}{\partial y_{k}\partial y_{r}}(0)\int_{\mathbb{R}^{m}}\left(\frac{U'(z)}{|z|}\right)^{2}z_{i}^{2}dz\\
 & -\frac{\varepsilon}{4}\tilde{a}(0)\eta_{r}\frac{\partial^{2}g^{ss}}{\partial y_{k}\partial y_{r}}\int_{\mathbb{R}^{m}}|\nabla U(z)|^2dz+o(\varepsilon)
\end{align*}
Here we  recognize the derivative of $I_1$ with respect to $\eta_k.$\\

We argue in a similar way for the all the other addenda in
 $J_{\varepsilon}\left(W_{\delta_{\varepsilon}(t),\eta}\right)$.
The claim follows easily.
\end{proof}

\end{proof}

 \begin{proof}[Proof of the Theorem \ref{main}]
 Set
 $ {\Theta(\xi_0)}:= h(\xi_{0}) -\frac{m-2}{4(m-1)}S_{g}(\xi_{0})   +{3m(m-2)\over 2(m-1)}\frac{\Delta_g \omega(\xi_0)}{\omega(\xi_{0})} .$
 Let $$\hbox{either}\ t_0: ={b_m\over\Theta(\xi_0)} \ \hbox{if}\ \varepsilon>0\ \hbox{or}\ t_0: =-{b_m\over\Theta(\xi_0)} \ \hbox{if}\ \varepsilon<0.$$
Since $\xi_0$ is a non degenerate critical point of $\omega,$ a straightforward computation shows that
the point $(t_0,0)$ is a non degenerate critical point of the function $\Phi$.   Therefore, by (ii) of Proposition \ref{lem:4.1}, if $\eps$ is small enough there exists $(t_\eps,\eta_\eps)\in(0,+\infty)\times\mathbb R^m$ which is a critical point of $\widetilde J_\eps $ such that $(t_\eps,\eta_\eps)\to (t_0,0)$ as $\eps\to0.$
Finally, by (i) of Proposition \ref{lem:4.1}, we deduce that $W_{\delta_{\varepsilon}(t ),\eta }+\Phi_{\delta_{\varepsilon}(t  ),\eta  }$
is a solution of   problem
(\ref{equ5})  which blows-up at the point $\xi_0$ as $\eps\to0.$
\end{proof}

\end{document}